\newif{\ifarxiv}
\arxivtrue

\documentclass[preprint,12pt]{elsarticle}
\pdfoutput=1 

\usepackage{enumitem}
\usepackage{latexsym,amsmath,amssymb,stmaryrd,graphicx}
\usepackage{color}
\usepackage[all]{xy}
\usepackage{url}

\newcommand\eqdef{\mathrel{\buildrel \text{def}\over=}}

\newcommand\pow{\mathbb{P}}
\newcommand\Pfin{\pow_{\mathrm{fin}}}

\newcommand{\real}{\mathbb{R}}

\newcommand\diff{\smallsetminus}
\DeclareMathOperator{\upc}{\uparrow\!}
\DeclareMathOperator{\dc}{\downarrow\!}
\newcommand\nat{\mathbb{N}}

\newcommand\uuarrow{\rlap{$\uparrow$}\raise.5ex\hbox{$\uparrow$}}
\newcommand\ddarrow{\rlap{$\downarrow$}\raise.5ex\hbox{$\downarrow$}}

\newcommand\identity[1]{\mathrm{id}_{#1}}
\newcommand\dcup{\bigcup\nolimits^{\scriptstyle\uparrow}}

\newcommand\Plotkin{\mathop{\mathcal P\ell}}
\newcommand\Plotkinn{\Plotkin_{\mathcal V}}
\newcommand\PV\Plotkinn 

\newcommand\Topcat{\mathbf{Top}}

\newcommand{\interior}[1]{\text{int} ({#1})} 

\newcommand\cb[1]{\mathbf{#1}} 
\newcommand{\catc}{{\cb{C}}}

\newcommand{\cati}{{\cb{I}}}

\newtheorem{theorem}{Theorem}[section]
\newtheorem{proposition}[theorem]{Proposition}

\newtheorem{corollary}[theorem]{Corollary}
\newtheorem{lemma}[theorem]{Lemma}
\newproof{proof}{Proof}

\newtheorem{fact}[theorem]{Fact}

\newtheorem{example}{Example}[section]

\newtheorem{remark}[theorem]{Remark}

\newcommand\ForAuthors[1]
 {\par\smallskip                     
  \begin{center}
   \fbox
   {\parbox{0.9\linewidth}
    {\raggedright--- #1}
   }
  \end{center}
  \par\smallskip                     %
 }

\journal{Topology and its Applications}

\begin{document}

\begin{frontmatter}



\title{A Few Projective Classes of (Non-Hausdorff) Topological Spaces}


\author{Jean Goubault-Larrecq}

\address{Universit\'e Paris-Saclay, CNRS, ENS Paris-Saclay,
  Laboratoire M\'ethodes
  Formelles, 91190, Gif-sur-Yvette, France.\\
  \texttt{jgl@lml.cnrs.fr}
}

\begin{abstract}
  A class of topological spaces is projective (resp.,
  $\omega$-projective) if and only if projective systems of spaces
  (resp., with a countable cofinal subset of indices) in the class are
  still in the class.  A certain number of classes of Hausdorff spaces
  are known to be, or not to be, ($\omega$-) projective.  We examine
  classes of spaces that are not necessarily Hausdorff.  Sober and
  compact sober spaces form projective classes, but most classes of
  locally compact spaces are not even $\omega$-projective.  Guided by
  the fact that the stably compact spaces are exactly the locally
  compact, strongly sober spaces, and that the strongly sober spaces
  are exactly the sober, coherent, compact, weakly Hausdorff (in the
  sense of Keimel and Lawson) spaces, we examine which classes defined
  by combinations of those properties are projective.  Notably, we
  find that coherent sober spaces, compact coherent sober spaces, as
  well as (locally) strongly sober spaces, form projective classes.
\end{abstract}


\end{frontmatter}


\noindent
\begin{minipage}{0.25\linewidth}
  \includegraphics[scale=0.2]{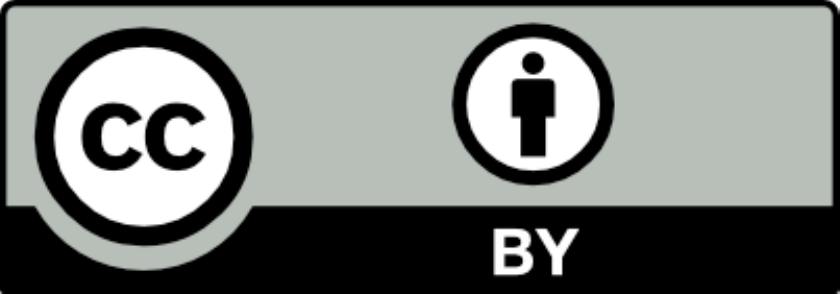}
\end{minipage}
\begin{minipage}{0.74\linewidth}
  \scriptsize
  For the purpose of Open Access, a CC-BY public copyright licence has
  been applied by the authors to the present document and will be
  applied to all subsequent versions up to the Author Accepted
  Manuscript arising from this submission.
\end{minipage}

\section{Introduction}
\label{sec:intro}

Projective limits, also called inverse limits, are categorical limits
of diagrams indexed by $(I, \sqsupseteq)$, where $(I, \sqsubseteq)$ is
a directed poset.  When $I$ has a countable cofinal subset, we will
call them $\omega$-projective limits.  Projective limits and
$\omega$-projective limits are an important concept in mathematics,
and this paper is concerned with exploring ($\omega$-)projective
limits in the category $\Topcat$ of topological spaces.  We call a
class of topological spaces ($\omega$-)\emph{projective} if and only
if it closed under ($\omega$-)projective limits.

Various classes of spaces are known to be projective or
$\omega$-projective.  The question is pretty simple for classes of
Hausdorff spaces, since in that case limits are closed subspaces of
products.  Hence every class of Hausdorff spaces that is closed under
homeomorphisms, closed subspaces and products is closed under
arbitrary limits and in particular forms a projective class; this
includes Hausdorff spaces, compact Hausdorff spaces, or regular
Hausdorff spaces for example.  Classes of spaces that are closed under
homeomorphisms, arbitrary subspaces and products are projective, too.
Classes of spaces that are closed under homeomorphisms, arbitrary
subspaces and countable products are $\omega$-projective, and this
includes first-countable spaces, second-countable spaces, and
metrizable spaces.  Finally, classes of Hausdorff spaces that are
closed under homeomorphisms, closed subspaces and countable products
are $\omega$-projective, for example completely metrizable spaces, and
Polish spaces.

The impetus of the present paper comes from \cite{JGL:kolmogorov},
where a few theorems in the style of Prokhorov's famous theorem on the
existence of projective limits of probability measures are shown for
continuous valuations (instead of measures), and for projective limits
of locally compact sober spaces, or for projective limits of
LCS-complete spaces (namely, $G_\delta$-subspaces of locally compact
sober spaces), notably.  Those are classes of spaces that are not
necessarily Hausdorff, and therefore we will focus on such classes of
non-Hausdorff spaces.  The importance of these classes stems from
domain theory \cite{AJ:domains}, where domains, or continuous dcpos,
are certain locally compact sober spaces.  Classical properties of
interest include sobriety, coherence, stable (local) compactness,
among others \cite{JGL-topology,GHKLMS:contlatt}; see
Section~\ref{sec:preliminaries}.  One of our guiding principles in
choosing which properties to examine is by picking one of the most
useful ones, stable compactness, and to examine how much of it is
preserved through projective limits.  As we will recall, the locally
compact, strongly sober spaces, and the strongly sober spaces are
exactly the sober, coherent, compact, weakly Hausdorff (in the sense
of Keimel and Lawson) spaces, and we will systematically inquire which
of those properties are preserved by ($\omega$-)projective limits.

\paragraph{Outline}
We recall a few preliminaries in Section~\ref{sec:preliminaries},
which we conclude by making a short tour of what is already known on
the topic, and in particular stating a fundamental theorem which we
call Steenrod's theorem, after Fujiwara and Kato
\cite{FK:rigid:geometry}.  Section~\ref{sec:sobr-comp-local} studies
the projectivity of classes defined by properties of sobriety,
compactness, and local compactness: sober and compact sober spaces
will form projective classes, but local compactness is usually
destroyed by taking projective
limits.  
In Section~\ref{sec:coherence}, we show that coherent sober
spaces form a projective class, and in
Section~\ref{sec:local-strong-sobr}, that (locally) strongly sober
spaces also form a projective class.  We conclude in
Section~\ref{sec:conclusion}.

\section{Preliminaries}
\label{sec:preliminaries}

A \emph{diagram} in a category $\catc$ is a functor
$F \colon \cati \to \catc$ from a small category $\cati$ to $\catc$.
We let $|\cati|$ denote the set of objects of $\cati$.  A \emph{cone}
of $F$ is a pair $X, {(p_i)}_{i \in |\cati|}$, where $X$ is an object
of $\catc$ and the morphisms $p_i \colon X \to F (i)$, for each
$i \in |\cati|$ are such that for every morphism
$\varphi \colon j \to i$ in $\cati$, $F (\varphi) \circ p_j = p_i$.  A
\emph{limit} of $F$ is a \emph{universal cone} of $F$, namely a cone
such that for every cone $Y, {(q_i)}_{i \in |\cati|}$ of $F$, there is
a unique morphism $f \colon Y \to X$ such that $p_i \circ f = q_i$ for
every object $i$ of $\cati$.  Limits are unique up to isomorphism when
they exist.  All limits exist in $\Topcat$, and the following is the
\emph{canonical limit} of $F$: $X$ is the subspace of
$\prod_{i \in |\cati|} F (i)$ consisting of those tuples $\vec x$ such
that $F (\varphi) (x_j) = x_i$ for every morphism
$\varphi \colon j \to i$ in $\cati$, with $p_i$ mapping $\vec x$ to
$x_i$.  We routinely write $\vec x$ for tuples
${(x_i)}_{i \in |\cati|}$, and $x_i$ for their $i$th components.

The special case of a diagram over the opposite $(I, \sqsupseteq)$ of
a directed preordered set $(I, \sqsubseteq)$ is called a
\emph{projective system}.  A \emph{directed} preordered set, or in
general a directed subset of a preordered set, is a non-empty subset
in which any two elements have an upper bound.  We call (canonical)
\emph{projective limit} any limit (the canonical limit) of a
projective system.  Explicitly, a projective system of topological
spaces, which we will write as
${(p_{ij} \colon X_j \to X_i)}_{i \sqsubseteq j \in I}$, is a
collection of spaces $X_i$ indexed by a directed preordered set
$(I, \sqsubseteq)$, with morphisms $p_{ij} \colon X_j \to X_i$ for all
indices $i \sqsubseteq j$ such that $p_{ii} = \identity {X_i}$ and
$p_{ij} \circ p_{jk} = p_{ik}$ for all $i \sqsubseteq j \sqsubseteq k$
in $I$.  Its canonical projective limit $X, {(p_i)}_{i \in I}$ is
given by
$\{\vec x \in \prod_{i \in I} X_i \mid \forall i \sqsubseteq j \in I,
p_{ij} (x_j) = x_i\}$, with the subspace topology from the product,
and where $p_i$ is projection onto coordinate $i$.  We will familiarly
call the maps $p_{ij}$ the \emph{bonding maps}.

When $I$ has a countable cofinal subset, we talk about
$\omega$-projective systems and $\omega$-projective limits.  The
latter are free from certain apparent pathologies: for example,
projective limits may be empty, even when every space $X_i$ is
non-empty and the maps $p_{ij}$ are surjective
\cite{Henkin:invlimit,Waterhouse:projlim:empty}, but this is not the
case with $\omega$-projective limits.


For background on topology, we refer the reader to
\cite{JGL-topology}, for the most part.  We write $\interior A$ for
the interior of $A$, $cl (A)$ for the closure of $A$.

Every topological space $X$ has a specialization preordering, which we
will ordinarily write as $\leq$, and defined by: $x \leq y$ if and
only if every open neighborhood of $x$ contains $y$.  We will also say
that $x$ is \emph{below} $y$ and that $y$ is \emph{above} $x$.  A
space is $T_0$ if and only if $\leq$ is antisymmetric, $T_1$ if and
only if $\leq$ is the equality relation.

An \emph{irreducible} closed subset $C$ of $X$ is a non-empty closed
subset such that, for any two closed subsets $C_1$ and $C_2$ of $X$
such that $C \subseteq C_1 \cup C_2$, $C$ is included in $C_1$ or in
$C_2$ already; equivalently, if $C$ intersects two open sets, it must
intersect their intersection.  A space $X$ is \emph{sober} if and only
if it is $T_0$ and every irreducible closed subset is of the form
$\dc x$ for some point $x \in X$.  The notation $\dc x$ stands for the
\emph{downward closure} of $x$ in $X$, namely the set of points $y$
below $x$.  Symmetrically, $\upc x$ stands for the \emph{upward
  closure} of $x$, namely the set of points $y$ above $x$.  This
notation extends to $\upc A$, for any subset $A$, denoting $\bigcup_{x
  \in A} \upc x$.

A subset of $X$ is \emph{saturated} if and only if it is equal to the
intersection of its open neighborhoods.  The saturated subsets are
exactly the upwards-closed subsets of $X$, namely the subsets that
contain any point above any of its elements.

A \emph{compact} subset is one such that one can extract a finite
subcover from any of its open covers.  No separation property is
assumed.  Without Hausdorffness, compactness is a pretty weak
property: for example, any space with a least element is compact.  For
every subset $A$ of a space $X$, $A$ is compact if and only if
$\upc A$ is, and then the latter is compact saturated.

A space $X$ is \emph{locally compact} if and only if every point has a
basis of compact saturated neighborhoods; i.e., for every $x \in X$,
for every open neighborhood $U$ of $x$ in $X$, there is a compact
saturated subset $Q$ of $X$ such that
$x \in \interior Q \subseteq Q \subseteq U$.  This coincides with the
usual definition in Hausdorff spaces (where we only require that each
point have a compact neighborhood), but the latter definition is
needed in non-Hausdorff spaces.  One should also note that a compact
space need not be locally compact, unlike what happens with Hausdorff
spaces.

It is traditional in domain theory to call \emph{coherent} any space
in which the intersection of any two compact saturated subsets is
compact (and necessarily saturated), and this is the convention we
shall use.  We should warn the reader that, in algebraic geometry,
coherent has another meaning.  For example, Fujiwara and Kato
\cite[Definition~2.2.1]{FK:rigid:geometry} define coherent as a
compact spaces with a base of compact-open subsets that is closed
under binary intersections; the coherent sober spaces in their sense
are also called \emph{spectral}.  All of them are coherent in our
sense, but the converse fails.  For example, any Hausdorff space is
coherent in our sense, say $\real$, but $\real$ is not spectral.

A \emph{stably locally compact space} is a sober, locally compact and
coherent space.  It is \emph{stably compact} if compact, too.  (Every
spectral space is stably compact.)

It turns out the the stably (locally) compact spaces are exactly the
locally compact, (locally) strongly sober spaces
\cite[Proposition~VI-6.15, Corollary~VI-6.16]{GHKLMS:contlatt}.  A
space $X$ is \emph{locally strongly sober} (resp., \emph{strongly
  sober}) if and only if the set $\lim \mathcal U$ of limits of every
convergent (resp., arbitrary) ultrafilter $\mathcal U$ on $X$ is the
closure $\dc x$ of a unique point $x$
\cite[Definition~VI-6.12]{GHKLMS:contlatt}.  A space is strongly sober
if and only if it is locally strongly sober and compact.

As we will see, projective limits of stably (locally) compact spaces
fail to be stably (locally) compact, but the only property that is
lost in the process is local compactness: (local) strong sobriety is
preserved.  We will use the following notions from \cite{JGL:wHaus} in
order to see this.

A \emph{weakly Hausdorff space}, in the sense of Keimel and Lawson
\cite[Lemma~6.6]{KL:measureext}, is a space in which any of the
following equivalent conditions is met:
\begin{itemize}
\item for any two compact saturated subsets $Q$ and $Q'$, for every
  open neighborhood $W$ of $Q \cap Q'$, there is an open neighborhood
  $U$ of $Q$ and an open neighborhood $V$ of $Q'$ such that $U \cap V
  \subseteq W$ (resp., $U \cap V = W$);
\item for any two points $x$ and $y$, for every open neighborhood $W$
  of $\upc x \cap \upc y$, there is an open neighborhood $U$ of $x$
  and an open neighborhood $V$ of $y$ such that $U \cap V \subseteq W$
  (resp., $U \cap V = W$).
\end{itemize}
As is mentioned in \cite[Proposition~2.2]{JGL:wHaus}, all Hausdorff
spaces are weakly Hausdorff, and in fact, the Hausdorff spaces are
exactly the $T_1$ weakly Hausdorff spaces; all stably compact spaces
are weakly Hausdorff; and all preordered sets are weakly Hausdorff in
their \emph{Alexandroff topology}, the topology whose open sets are
the upwards-closed subsets.

Then a space is locally strongly sober if and only if it is sober,
coherent, and weakly Hausdorff \cite[Theorem~3.5]{JGL:wHaus}.

Although that will not be central to this work, we should define dcpos
(directed-complete partial orders).  Those are posets in which every
directed family has a supremum.  We will always equip such posets with
the Scott topology.  The latter is defined on all posets, as follows:
the (Scott-)\emph{open} subsets of a poset $P$ are the upwards-closed
subsets $U$ such that for every directed family $D$ whose supremum
exists and is in $U$, $D$ intersects $U$.

\paragraph{Related work}
The most closely related result to this work is \emph{Steenrod's
  theorem}, as stated by Fujiwara and Kato
\cite[Theorem~2.2.20]{FK:rigid:geometry}: every projective limit,
taken in $\Topcat$, of compact sober spaces is compact.  We will see
that this implies that the class of compact sober spaces is
projective.  Fujiwara and Kato apparently call this Steenrod's theorem
because of a similar claim by Steenrod
\cite[Theorem~2.1]{Steenrod:univ:homol}, see the discussion at the
beginning of Section~7 of \cite{JGL:kolmogorov}.  A very useful lemma
that comes naturally with that result is the following, which appears
as Lemma~7.5 in \cite{JGL:kolmogorov}.
\begin{lemma}
  \label{lemma:steenrod:open}
  Let $Q, {(p_i)}_{i \in I}$ be the canonical projective limit of a
  projective system
  ${(p_{ij} \colon Q_j \to Q_i)}_{i \sqsubseteq j \in I}$ of compact
  sober spaces.  For every $i \in I$, for every open neighborhood $U$
  of $\upc p_i [Q]$ in $Q_i$, there is an index $j \in I$ such that
  $i \sqsubseteq j$ and $\upc p_{ij} [Q_j] \subseteq U$.
\end{lemma}

There are other papers on the question of projective classes, for
example
\cite{Nagami:invlim:paracompact,Nogura:invlim:frechet,FP:invlim:min},
but all are concerned with classes of Hausdorff spaces; Camargo and
Uzc\'ategui \cite{CU:invlim:q+} are concerned about
$\omega$-projective limits of specific classes of regular spaces.

We will exclusively be concerned with projective systems whose bonding
maps are continuous.  Stone establishes theorems on projective limits
of compact $T_1$ spaces and \emph{closed} continuous bonding maps
\cite{Stone:limproj:compact}, and that is different.  (We will borrow
Example~\ref{exa:Stone} from Stone, though.)  Apart from stating and
proving Steenrod's theorem, Fujiwara and Kato \cite[Proposition~2.2.9,
Theorem~2.2.10]{FK:rigid:geometry} focus on projective systems of
spectral spaces and \emph{spectral} maps (i.e., maps such that the
inverse image of a compact-open set is compact-open).  We mention the
latter because they call spectral spaces ``coherent sober'' and the
theorems just cited state that coherent sober spaces form a projective
class, which may seem to be exactly what we say in
Section~\ref{sec:coherence}.  But, first, coherent has a different
meaning there, and second, they assume the bonding maps to be
spectral, while we only require them to be continuous.

One should also mention a very special case of projective systems that
occurs frequently in denotational semantics, that of ep-systems.  They
form the basic tool for Scott's celebrated construction of the
$D_\infty$ model of the pure $\lambda$-calculus \cite{scott69} and for
Plotkin's bifinite domains \cite{plotkin81}, and are the key to the
construction of recursive domain equations, giving dcpo semantics to
data types in higher-order programming languages
\cite[Section~3.3]{AJ:domains}.  An \emph{embedding-projection pair},
or \emph{ep-pair} for short, is a pair $(e, p)$ of continuous maps
$\xymatrix{X \ar@<1ex>[r]^e & Y \ar@<1ex>[l]^p}$ such that
$p \circ e = \identity X$ and $e \circ p \leq \identity Y$.  Then $p$
is called a \emph{projection} of $Y$ onto $X$, and $e$ is the
associated \emph{embedding}.  In general, we call a continuous map
$p \colon Y \to X$ a \emph{projection} if and only if it is the
projection part of some ep-pair.  If that is the case, and if $Y$ is
$T_0$, then the associated embedding $e \colon X \to Y$ must be such
that, for every $x \in X$, $e (x)$ is the smallest element $y \in Y$
such that $x \leq p (y)$, so that $e$ is determined uniquely from $p$.
We will call \emph{ep-system} any projective system in the category of
topological spaces and ep-pairs.  When every $X_i$ is $T_0$, this is
equivalent to an ordinary projective system
${(p_{ij} \colon X_j \to X_i)}_{i \sqsubseteq j \in I}$ where all
bonding maps $p_{ij}$ are projections, since the corresponding
embeddings are determined uniquely.  Some of our counter-examples will
not just be projective systems, but ep-systems
(Example~\ref{exa:knijnenburg:wH}, Example~\ref{exa:knijnenburg:coh}).

\section{Sobriety, compactness, local compactness}
\label{sec:sobr-comp-local}

Any limit of sober spaces, taken in $\Topcat$, is sober \cite[Theorem
8.4.13]{JGL-topology}.  In particular, the class of sober spaces is
projective.

Combining this with Steenrod's theorem, we obtain the following.
\begin{fact}
  \label{fact:steenrod}
  The class of compact sober spaces is projective.
\end{fact}

Despite this, compact spaces do not form a projective class, and not
even an $\omega$-projective class.  The following counterexample is
due to Stone.
\begin{example}[Example~3 of \cite{Stone:limproj:compact}]
  \label{exa:Stone}
  We let $X_n$ be $\nat$ for every natural number $n$ and
  $p_{mn} \colon X_n \to X_m$ be the identity map for all $m \leq n$.
  The topology on $X_n$ is obtained by declaring a subset $C$ closed
  if and only if $C \cap \{n, n+1, \cdots\}$ is finite or equal to the
  whole of $\{n, n+1, \cdots\}$.  In other words, $X_n$ is isomorphic
  to the disjoint sum of $\{0, 1, \cdots, n\}$ with the discrete
  topology with $\{n, n+1, \cdots\}$ with the cofinite topology.  Then
  each $X_n$ is compact, but the projective limit is $\nat$ with the
  discrete topology, which is not.
\end{example}

Alongside with Stone, we note that each of the spaces $X_n$ is even
\emph{Noetherian}, namely it satisfies any one of the following
equivalent properties \cite[Section~9.7]{JGL-topology}: it is
hereditarily compact, namely every subspace is compact; every open
subset is compact; every monotonic chain of open subsets
$U_0 \subseteq U_1 \subseteq \cdots \subseteq U_n \subseteq \cdots$
stabilizes (all the sets $U_n$ are equal for $n$ large enough); there
is no infinite strictly decreasing sequence of closed subsets.

\begin{remark}
  \label{rem:id:bond}
  Stone's counterexample has a particular form: all the bonding maps
  are identity maps.  Any projective system
  ${(p_{ij} \colon X_j \to X_i)}_{i \sqsubseteq j \in I}$ where each
  $p_{ij}$ is the identity map, and where therefore each $X_i$ is the
  same set $E$, with a possibly different topology $\tau_i$, has a
  projective limit which is just $E$ itself, with topology
  $\bigvee_{i \in I} \tau_i$, the coarsest topology finer than every
  $\tau_i$.  (The cone maps $p_i \colon E \to X_i$ are identity maps.)
  In brief, directed suprema of topologies are projective limits.
\end{remark}

We now deal with local compactness.  We will use the following remark.
\begin{remark}
  \label{rem:prod}
  In a category $\catc$ with finite products, a product
  $\prod_{i \in I} X_i$ is the same thing as a projective limit of
  ${(p_{JK} \colon \prod_{i \in K} X_i \to \prod_{i \in J} X_i)}_{J
    \subseteq K \in \Pfin (I)}$, where $\Pfin (I)$ denotes the set of
  finite subsets of $I$, and $p_{JK}$ is the obvious map.
\end{remark}

We recall that Baire space $\nat^\nat$ is the product of countably
many copies of $\nat$, each one with the discrete topology.  It is
well-known that $\nat^\nat$ is not locally compact, because the
interior of every compact subset is empty
\cite[Example~4.8.12]{JGL-topology}.
\begin{proposition}
  \label{prop:Baire}
  The classes of locally compact spaces, of locally compact sober
  spaces, of stably locally compact spaces, of locally compact
  Hausdorff spaces, are not $\omega$-projective.
\end{proposition}
\begin{proof}
  Every finite power $\nat^k$ of $\nat$ is a discrete space, and as
  such is locally compact Hausdorff, hence in particular sober
  \cite[Proposition~8.2.12]{JGL-topology}.  It is also coherent, since
  its compact subsets are its finite subsets, hence it is stably
  locally compact.  But $\nat^\nat$ is not locally compact, and we
  conclude by Remark~\ref{rem:prod}.  \qed
\end{proof}

Compactness does not help.  We rely on the following construction: the
\emph{lifting} $X_\bot$ of a space $X$ is obtained by adjoining a
fresh point $\bot$ (``bottom'') to $X$, and declaring that the open
subsets of $X_\bot$ are those of $X$ plus $X_\bot$ itself.  $X_\bot$
is always a compact space \cite[Exercise~4.4.25]{JGL-topology}.  The
specialization preordering of $X_\bot$ is given by $x \leq y$ if and
only if $x \leq y$ in $X$ or $x=\bot$, the compact saturated subsets
of $X_\bot$ are those of $X$ plus $X_\bot$ itself
\cite[Exercise~4.4.26]{JGL-topology}.  As a result, $X_\bot$ is
locally compact if and only if $X$ is \cite[Exercise
4.8.6]{JGL-topology}.  Also, $X_\bot$ is sober, resp.\ coherent,
resp.\ stably locally compact if and only if $X$ is \cite[Exercise
8.2.9, Proposition~9.1.10]{JGL-topology}.



Lifting defines a functor $\__\bot$, whose application to any
continuous map $f \colon X \to Y$ yields
$f_\bot \colon X_\bot \to Y_\bot$ defined as $f$ on $X$ and mapping
$\bot$ to $\bot$.  We say that a category $\cati$ is \emph{connected}
if and only if it has at least one object, and every two objects are
connected by a zig-zag of morphisms; namely if and only if the
smallest equivalence relation $\sim$ on $|\cati|$ such that $i \sim j$
if there is a morphism from $i$ to $j$ in $\cati$ makes all objects
equivalent.
A \emph{connected limit} is the limit of a diagram over a connected
category $\cati$.
\begin{lemma}
  \label{lemma:projlim:lift}
  Lifting preserves connected limits, in particular projective limits.
\end{lemma}
\begin{proof}
  Let $F \colon \cati \to \Topcat$ be a diagram, where $\cati$ is
  connected, and let us write $X_i$ for $F (i)$, $i \in |\cati|$.  Let
  $X, {(p_i)}_{i \in |\cati|}$ be a limit of $F$.  It is clear that
  $X_\bot, {(p_{i\bot})}_{i \in |\cati|}$ is a cone of
  $\__\bot \circ F$.

  Let $Y, {(q_i)}_{i \in I}$ be any cone of $\__\bot \circ F$.  For
  every $i \in |\cati|$, $X_i$ is open in $X_{i\bot}$, so
  $q_i^{-1} (X_i)$ is an open subset of $Y$.  We claim that it is
  independent of $i$.  For every morphism $\varphi \colon j \to i$ in
  $\cati$, we have $F (\varphi)_\bot \circ q_j = q_i$, so
  $q_i^{-1} (X_i) = q_j^{-1} (F (\varphi)_\bot^{-1} (X_i))$; by
  definition of $\__\bot$, $F (\varphi)_\bot^{-1} (X_i) = X_j$, so
  $q_i^{-1} (X_i) = q_j^{-1} (X_j)$.  Since $\cati$ is connected, it
  follows that $q_i^{-1} (X_i) = q_j^{-1} (X_j)$ for every pair of
  objects $i$ and $j$ of $\cati$.

  Since $\cati$ has at least one object, it makes sense to define $V
  \eqdef q_i^{-1} (X_i)$, where $i$ is any arbitrary object of
  $\cati$.

  Any continuous map $f \colon Y \to X_\bot$ such that
  $p_{i\bot} \circ f = q_i$ for every object $i$ of $\cati$ must be
  such that
  $V = q_i^{-1} (X_i) = f^{-1} (p_{i\bot} (X_i)) = f^{-1} (X)$.  Hence
  $f$ must map $V$ to $X$, and since $p_{i\bot}$ coincides with $p_i$
  on $V$, the restriction $f_{|V}$ must be such that
  $p_i \circ f_{|V} = q_i$ for every $i \in |\cati|$.  By the
  universal property of the limit $X, {(p_i)}_{i \in |\cati|}$, this
  determines $f_{|V}$ uniquely; and $f (y)$ must be equal to $\bot$
  for every $y \in Y \diff V$.  This shows the uniqueness of $f$.

  As for existence, let us consider $V$ as a subspace of $Y$.  The
  restrictions $q_{i|V} \colon V \to X_{i\bot}$ have images included
  in $X_i$, by definition of $V$ as $q_i^{-1} (X_i)$, and we will
  consider them as maps from $V$ to $X_i$.  For every open subset $U$
  of $X_i$, $U$ is open in $X_{i\bot}$, too, and then
  $q_{i|V}^{-1} (U) = q_i^{-1} (U) \cap V$ is open in $V$.  For every
  morphism $\varphi \colon j \to i$ in $\cati$,
  $F (\varphi) \circ q_{j|V}$ maps every point $y$ of $V$ to
  $F (\varphi) (q_j (y))$, which is also equal to
  $F (\varphi)_\bot (q_j (y))$, hence to $q_i (y) = q_{i|V} (y)$.
  Therefore $V, {(q_{i|V})}_{i \in |\cati|}$ is a cone over the
  diagram $F$.  By the universal property of
  $X, {(p_i)}_{i \in |\cati|}$, there is a (unique) continuous map
  $g \colon V \to X$ such that $p_i \circ g = q_{i|V}$ for every
  $i \in |\cati|$.  We define $f \colon Y \to X_\bot$ so that
  $f (y) \eqdef g (y)$ for every $y \in V$ and $f (y) \eqdef \bot$ for
  every $y \in Y \diff V$.  The map $f$ is continuous:
  $f^{-1} (X_\bot)=Y$ is obviously open, and for every open subset $U$
  of $X_\bot$ other than $X_\bot$ itself, $U$ is an open subset of
  $X$, and then $f^{-1} (U) = g^{-1} (U)$ is open in $V$, hence in
  $Y$, since $V$ is itself open in $Y$.
  It is easy to see that $p_{i\bot} \circ f = q_i$ for every $i \in
  |\cati|$, which concludes the proof.  \qed
\end{proof}

\begin{corollary}
  \label{corl:Baire:lift}
  The classes of compact locally compact spaces, of compact locally
  compact sober spaces, of stably compact spaces, are not
  $\omega$-projective.
\end{corollary}
\begin{proof}
  By Remark~\ref{rem:prod} and Lemma~\ref{lemma:projlim:lift},
  ${(\nat^\nat)}_\bot$ is the projective limit of spaces of the form
  ${(\nat^k)}_\bot$, which are all stably compact, since $\nat^k$ is
  stably locally compact.  The only compact subset $K$ of
  ${(\nat^\nat)}_\bot$ with non-empty interior is the whole space: all
  others are included in $\nat^\nat$, and therefore are compact
  subsets of Baire space.  Hence ${(\nat^\nat)}_\bot$ is not locally
  compact.  \qed
\end{proof}

\section{Coherence}
\label{sec:coherence}

We will see that coherent sober spaces form a projective class.  We
first need the following remarks.

\begin{remark}
  \label{rem:sat:sober}
  Any saturated subset of a sober space $X$ is sober, when seen as a
  subspace.  Indeed, the sober subspaces of $X$ coincide with the
  subsets that are closed in the strong topology
  \cite[Corollary~3.5]{KL:dcompl}.  The latter is also known as the
  Skula topology, and is the smallest one generated by the original
  topology on $X$ and all the downwards-closed subsets.  In
  particular, any saturated subset of $X$ is upwards-closed, hence the
  complement of a downwards-closed subset, hence closed in the strong
  topology.
\end{remark}

\begin{remark}
  \label{rem:prohorov:top}
  Let ${(p_{ij} \colon X_j \to X_i)}_{i \sqsubseteq j \in I}$ be a
  projective system of topological spaces, with projective limit
  $X, {(p_i)}_{i \in I}$, and let $U$ be any open subset of $X$.  For
  every $i \in I$, there is a largest open subset $U_i$ of $X_i$ such
  that $p_i^{-1} (U_i) \subseteq U$.  Then, for all $i \sqsubseteq j$,
  $p_i^{-1} (U_i) \subseteq p_j^{-1} (U_j)$, so the union
  $\bigcup_{i \in I} p_i^{-1} (U_i)$ is directed---we will write
  $\dcup_{i \in I} p_i^{-1} (U_i)$ to stress this---, and
  $U = \dcup_{i \in I} p_i^{-1} (U_i)$
  \cite[Lemma~3.1]{JGL:kolmogorov}.
\end{remark}

The following lemma essentially says that the compact saturated
subsets of a projective limit $X$ of sober spaces $X_i$ are simply the
projective limits of compact saturated subsets of each $X_i$.  We
consider each subset $Q_i$ as a subspace, and we note that it is
compact as a subset if and only if it is compact as a subspace.
\begin{lemma}
  \label{lemma:limproj:compact}
  Let ${(p_{ij} \colon X_j \to X_i)}_{i \sqsubseteq j \in I}$ be a
  projective system of sober spaces and $X, {(p_i)}_{i \in I}$ be its
  canonical projective limit.  Then:
  \begin{enumerate}
  \item every canonical projective limit of
    ${(p_{ij|Q_j} \colon Q_j \to Q_i)}_{i \sqsubseteq j \in I}$, where
    each $Q_i$ is compact saturated in $X_i$ and
    $p_{ij} [Q_j] \subseteq Q_i$ for all $i \sqsubseteq j \in I$, is a
    compact saturated subset of $X$;
  \item every compact saturated subset $K$ of $X$ can be written as
    such a projective limit, where $Q_i \eqdef \upc p_i [K]$ for every
    $i \in I$, and then $\upc p_{ij} [Q_j] = Q_i$ for all
    $i \sqsubseteq j \in I$;
  \end{enumerate}
\end{lemma}
\begin{proof}
  1.  Let $Q_i$ be compact saturated subsets of $X_i$ for each
  $i \in I$ such that $p_{ij} [Q_j] \subseteq Q_i$ for all
  $i \sqsubseteq j \in I$.  Then $p_{ij|Q_j}$ is a continuous map from
  $Q_j$ to $Q_i$.  Also, since each $X_i$ is sober, its subspace $Q_i$
  is compact and sober by Remark~\ref{rem:sat:sober}.  By Steenrod's
  theorem (Fact~\ref{fact:steenrod}), its canonical projective limit
  $Q$ is a compact sober space.  Clearly, $Q$ is a subset of $X$, and
  we have just shown that it is a compact subset of $X$.

  2.  We fix an arbitrary compact saturated subset $K$ of $X$.  Let
  $Q_i \eqdef \upc p_i [K]$: $Q_i$ is compact saturated in $X_i$ for
  every $i \in I$, and we will show that $K$ is equal to the canonical
  projective limit of
  ${(p_{ij|Q_j} \colon Q_j \to Q_i)}_{i \sqsubseteq j \in I}$.

  We first need to verify that the latter is a family of maps from
  $Q_j$ to $Q_i$.  We show more, namely: $(*)$ for all
  $i \sqsubseteq j \in I$, $\upc p_{ij} [Q_j] = Q_i$.  In one
  direction,
  $\upc p_{ij} [Q_j] \supseteq \upc p_{ij} [p_j [K]] = \upc p_i [K] =
  Q_i$.  In the other direction, we show that every open neighborhood
  $U$ of $Q_i$ contains $p_{ij} [Q_j]$.  Since
  $Q_i = \upc p_i [K] \subseteq U$, $K$ is included in $p_i^{-1} (U)$,
  namely in $p_j^{-1} (p_{ij}^{-1} (U))$.  Therefore
  $p_i [K] \subseteq p_{ij}^{-1} (U)$.  Since $p_{ij}^{-1} (U)$ is
  open hence upwards-closed, it follows that
  $Q_j = \upc p_i [K] \subseteq p_{ij}^{-1} (U)$, hence that
  $p_{ij} [Q_j] \subseteq U$.

  Let $Q$ be the canonical projective limit of
  ${(p_{ij|Q_j} \colon Q_j \to Q_i)}_{i \sqsubseteq j \in I}$.  We
  note that each $Q_i$ is a compact space, seen as a subspace of
  $X_i$, and is sober because $X_i$ is, relying on
  Remark~\ref{rem:sat:sober}.  This will allow us to use
  Lemma~\ref{lemma:steenrod:open}.

  We claim that: $(**)$ $\upc p_i [Q] = Q_i$ for every $i \in I$.  By
  definition, $p_i [Q] \subseteq Q_i$, so
  $\upc p_i [Q] \subseteq Q_i$.  We show the reverse inclusion by
  showing that every open neighborhood $U$ of $\upc p_i [Q]$ in $X_i$
  contains $Q_i$.  Since $\upc p_i [Q] \subseteq U$ and
  $\upc p_i [Q] \subseteq Q_i$, $U \cap Q_i$ is an open neighborhood
  of $\upc p_i [Q]$ in $Q_i$.  By Lemma~\ref{lemma:steenrod:open},
  there is an index $j \in I$ such that $i \sqsubseteq j$ and
  $\upc p_{ij} [Q_j] \subseteq U \cap Q_i$.  But, by $(*)$,
  $\upc p_{ij} [Q_j] = Q_i$, so $Q_i \subseteq U$.
  
  For every $\vec x \in K$, for every $i \in I$, $x_i = p_i (\vec x)$
  is in $Q_i$, and $x_i = p_{ij} (x_j)$ for all
  $i \sqsubseteq j \in I$, since $\vec x$ is in $X$; so
  $\vec x \in Q$.  Hence $K \subseteq Q$.


  In order to show the reverse inclusion $Q \subseteq K$, it suffices
  to show that every open neighborhood $U$ of $K$ in $X$ contains $Q$.
  This is because $K$ is saturated.  Using
  Remark~\ref{rem:prohorov:top} and the notations used there,
  $U = \dcup_{i \in I} p_i^{-1} (U_i)$.  Since $K$ is compact, $K$ is
  included in $p_i^{-1} (U_i)$ for some $i \in I$.  Then
  $Q_i = \upc p_i [K]$ is included in $U_i$.  By $(**)$, it follows
  that $\upc p_i [Q]$ is included in $U_i$, so
  $Q \subseteq p_i^{-1} (U_i) \subseteq U$.  Therefore $Q=K$.  \qed
\end{proof}

\begin{theorem}
  \label{thm:projective:coh}
  The class of coherent sober spaces is projective.
\end{theorem}
\begin{proof}
  Let ${(p_{ij} \colon X_j \to X_i)}_{i \sqsubseteq j \in I}$ be a
  projective system of coherent sober spaces and
  $X, {(p_i)}_{i \in I}$ be its canonical projective limit.  We
  consider any two compact saturated subsets $Q$ and $Q'$ of $X$.
  Using Lemma~\ref{lemma:limproj:compact}, we write $Q$ as the
  canonical projective limit of
  ${(p_{ij|Q_j} \colon Q_j \to Q_i)}_{i \sqsubseteq j \in I}$, where
  each $Q_i$ is compact saturated in $X_i$ and
  $p_{ij} [Q_j] \subseteq Q_i$ for all $i \sqsubseteq j \in I$, and
  similarly with $Q'$, with compact saturated subsets $Q'_i$ of $X_i$.
  The elements $\vec x \eqdef {(x_i)}_{i \in I}$ of $Q \cap Q'$ are
  those elements of $X$ such that $x_i \in Q_i$ for every $i \in I$
  and $x_i \in Q'_i$ for every $i \in I$, equivalently those such that
  $x_i \in Q_i \cap Q'_i$ for every $i \in I$.  Hence $Q \cap Q'$ is
  the canonical projective limit of
  ${(p_{ij|Q_j \cap Q'_j} \colon Q_j \cap Q'_j \to Q_i \cap Q'_i)}_{i
    \sqsubseteq j \in I}$.  Since each $X_i$ is coherent, every
  intersection $Q_i \cap Q'_i$ is compact saturated, so $Q \cap Q'$ is
  compact saturated by Lemma~\ref{lemma:limproj:compact}.  \qed
\end{proof}

Sobriety is essential, as we now show.
\begin{example}
  \label{exa:projective:coh:no}
  We consider the following space $X$; that is the subspace of maximal
  elements of a dcpo built by X. Jia \cite[Example~4.1.2]{Jia:PhD} in
  order to show that the well-filteredness assumption is required in
  the Jia-Jung-Li theorem \cite[Lemma~3.1]{JJL:dcpo:coh}.  $X$
  consists of two disjoint copies of $\nat$, $X^-$ and $X^+$; we write
  the elements of $X^-$ (resp., $X^+$) as $n^-$ (resp., $n^+$),
  $n \in \nat$.  The closed subsets of $X$ are: $(a)$ the finite
  subsets of $X^- \cup X^+$, and $(b)$ any subset of $X^- \cup X^+$
  that contains $X^-$.  We let the reader check that $X$ is a $T_1$
  space whose compact subsets are the finite subsets of $X^+$ and the
  subsets of $X$ that intersect $X^-$.  In particular, the sets
  $Q \eqdef \{n^- \mid n \text{ odd}\} \cup X^+$ and
  $Q' \eqdef \{n^- \mid n \text{ even}\} \cup X^+$ are compact
  (saturated), but their intersection is not.  Hence $X$ is not
  coherent.
  
  We observe that $X$ arises as a projective limit of coherent spaces,
  using a construction similar to Stone's example
  (Example~\ref{exa:Stone}).  For each $n \in \nat$, let $X_n$ consist
  of the same points as $X$, but with the following modified topology
  $\tau_n$.  Its closed subsets are sets of type $(a)$ as given above,
  or $(b_n)$: sets of the form $X^- \cup A^+$, where $A^+$ is a subset
  of $X^+$ that is finite or contains every point $k^+$ with
  $k \geq n$.  Each space $X_n$ is Noetherian, hence coherent.  The
  supremum topology $\bigvee_{n \in \nat} \tau_n$ coincides with the
  topology of $X$.  By Remark~\ref{rem:id:bond}, we have therefore
  obtained the non-coherent space $X$ as a projective limit of an
  $\omega$-projective system of coherent spaces $X_n$.
\end{example}

\begin{corollary}
  \label{corl:coh:proj}
  The class of coherent spaces is not $\omega$-projective.
\end{corollary}
Using Lemma~\ref{lemma:projlim:lift}, $X_\bot$ is an
$\omega$-projective limit of the spaces $X_{n\bot}$, where $X$ and
$X_n$ are as in Example~\ref{exa:projective:coh:no}.  Hence:
\begin{corollary}
  \label{corl:coh:proj:compact}
  The class of coherent, compact spaces is not $\omega$-projective.
\end{corollary}

\section{Local strong sobriety}
\label{sec:local-strong-sobr}

We recall that the locally strongly sober spaces are the sober,
coherent, and weakly Hausdorff spaces \cite[Theorem~3.5]{JGL:wHaus}.

\begin{theorem}
  \label{thm:projective:lssober}
  The class of locally strongly sober spaces is projective, and so is
  the class of strongly sober spaces.
\end{theorem}
\begin{proof}
  The second part follows from the first part and Steenrod's theorem
  (Fact~\ref{fact:steenrod}), remembering that the strongly sober
  spaces are the locally strongly sober spaces that are compact.
  
  Let ${(p_{ij} \colon X_j \to X_i)}_{i \sqsubseteq j \in I}$ be a
  projective system of coherent, weakly Hausdorff sober spaces and
  $X, {(p_i)}_{i \in I}$ be its canonical projective limit.  Let $Q$
  and $Q'$ be two compact saturated subsets of $X$, and $W$ be an open
  neighborhood of $Q \cap Q'$.  By Theorem~\ref{thm:projective:coh},
  $X$ is coherent, so $Q \cap Q'$ is compact saturated.

  We use Remark~\ref{rem:prohorov:top}, and we write $W$ as
  $\dcup_{i \in I} p_i^{-1} (W_i)$, where $W_i$ is open in $X_i$ and
  $p_i^{-1} (W_i) \subseteq W$.  Since $Q \cap Q'$ is compact, there
  is an index $i \in I$ such that
  $Q \cap Q' \subseteq p_i^{-1} (W_i)$.

  As in the proof of Theorem~\ref{thm:projective:coh}, $Q \cap Q'$ is
  in fact the canonical projective limit
  ${(p_{ij|Q_j \cap Q'_j} \colon Q_j \cap Q'_j \to Q_i \cap Q'_i)}_{i
    \sqsubseteq j \in I}$, where $Q$ (resp.\ $Q'$) is the canonical
  projective limit of a projective system
  ${(p_{ij|Q_j} \colon Q_j \to Q_i)}_{i \sqsubseteq j \in I}$ of
  compact saturated subsets $Q_i$ (resp.\ $Q'_i$) of each $X_i$.
  Lemma~\ref{lemma:limproj:compact}, item~2, also tells us that we can
  take $Q_i \eqdef \upc p_i [Q]$ and $Q'_i \eqdef \upc p_i [Q']$.

  As saturated subspaces of the coherent sober spaces $X_i$, the
  spaces $Q_i \cap Q'_i$ are compact, and sober by
  Remark~\ref{rem:sat:sober}.  Since
  $Q \cap Q' \subseteq p_i^{-1} (W_i)$, $\upc p_i [Q \cap Q']$ is a
  subset of the open subset $W_i \cap (Q_i \cap Q'_i)$ of the space
  $Q_i \cap Q'_i$.  We can therefore use
  Lemma~\ref{lemma:steenrod:open}: there is an index $j \in I$ such
  that $i \sqsubseteq j$ and
  $\upc p_{ij} [Q_j \cap Q'_j] \subseteq W_i$.  Hence
  $Q_j \cap Q'_j \subseteq p_{ij}^{-1} (W_i)$.  Since $X_j$ is weakly
  Hausdorff, there is an open neighborhood $U$ of $Q_j$ and an open
  neighborhood $V$ of $Q'_j$ in $X_j$ such that
  $U \cap V \subseteq p_{ij}^{-1} (W_i)$.  We recall that
  $Q_j = \upc p_j [Q]$, so $Q \subseteq p_j^{-1} (U)$, and similarly
  $Q' \subseteq p_j^{-1} (V)$.  Additionally,
  $p_j^{-1} (U) \cap p_j^{-1} (V) = p_j^{-1} (U \cap V) \subseteq
  p_j^{-1} (p_{ij}^{-1} (W_i)) = p_i^{-1} (W_i) \subseteq W$.  Hence
  $X$ is weakly Hausdorff, and it is coherent and sober by
  Theorem~\ref{thm:projective:coh}.  \qed
\end{proof}

\begin{remark}
  \label{rem:limU}
  Another, apparently easier way of proving
  Theorem~\ref{thm:projective:lssober} would be to return to the
  original definition of (locally) strongly sober spaces.  Given a
  convergent ultrafilter $\mathcal U$ on the projective limit $X$, we
  form its images $p_i [\mathcal U]$ under $p_i \colon X \to X_i$.
  Those are convergent ultrafilters in each $X_i$, and since $X_i$ is
  assumed to be (locally) strongly sober, $\lim p_i [\mathcal U]$ is
  the downward closure $\dc x_i$ of a unique point $x_i$ in $X_i$.  It
  is easy to see that $C \eqdef \bigcap_{i \in I} p_i^{-1} (\dc x_i)$
  is exactly the set of limits of $\mathcal U$, but showing that this
  set has a largest element seems to be no easier than doing what we
  did to prove Theorem~\ref{thm:projective:lssober}.  Note that there
  is no reason to believe that ${(x_i)}_{i \in I}$ would even live in
  $X$.
\end{remark}

We will show that both coherence and sobriety are required in order
for weak Hausdorffness to be preserved by projective limits.  The basic
construction is the following.
\begin{figure}
  \centering
  \includegraphics[scale=0.32]{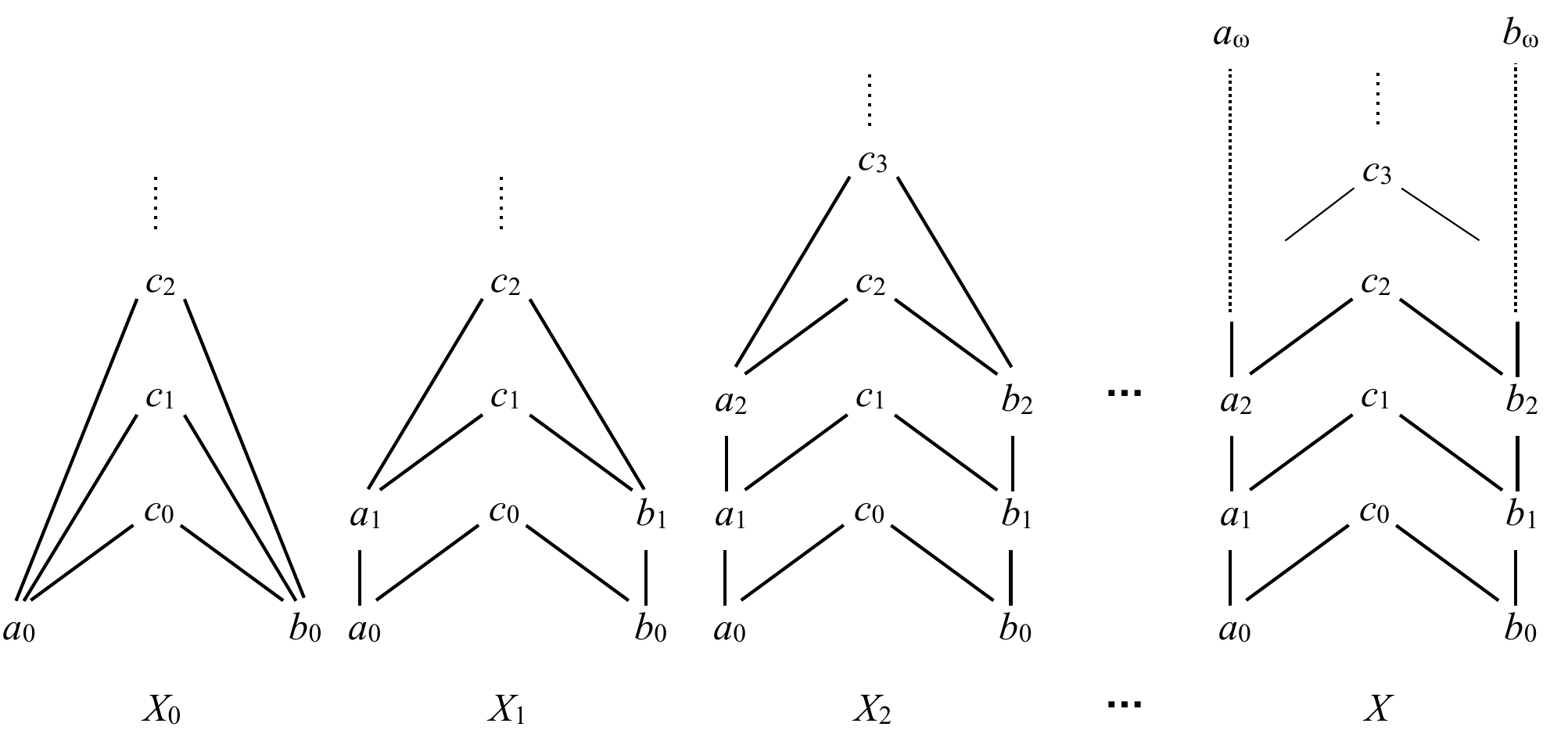}
  \caption{Knijnenburg's dcpo, as a projective limit}
  \label{fig:knijnenburg:limproj}
\end{figure}
\begin{example}
  \label{exa:knijnenburg}
  The dcpo $X$ shown on the right of
  Figure~\ref{fig:knijnenburg:limproj} was invented by Knijnenburg
  \cite[Example~6.1]{Knijnenburg:powerdomain} in order to show that
  the Egli-Milner ordering and the topological Egli-Milner ordering on
  spaces of lenses may fail to coincide.  (We have removed its bottom
  element, which serves no purpose here.)  Explicitly, the elements of
  $X$ are $a_n$ and $b_n$ for every $n \in \nat \cup \{\omega\}$, and
  $c_n$ for every $n \in \nat$ (not $\omega$), all pairwise distinct.
  The ordering is given by: $a_m \leq a_n$, $b_m \leq b_n$,
  $a_m \leq c_n$, $b_m \leq c_n$ if and only if $m \leq n$; all other
  pairs of elements are incomparable.  Note that the points $c_n$,
  $n \in \nat$, are all incomparable.  We give $X$ the Scott topology.
  One checks that neither $\{a_\omega\}$ nor $\{b_\omega\}$ is open,
  but that $\upc x$ is open for every
  $x \in X \diff \{a_\omega, b_\omega\}$, and those form a base of the
  Scott topology.  (In domain-theoretic terms, $X$ is an algebraic
  domain whose finite elements are all elements except $a_\omega$ and
  $b_\omega$.)  In particular, every open neighborhood of $a_\omega$
  intersects every open neighborhood of $b_\omega$, so $X$ is not
  weakly Hausdorff.
\end{example}

We observe that Knijnenburg's dcpo arises as an $\omega$-projective
limit in at least two different ways.  Here is the first way.

\begin{example}
  \label{exa:knijnenburg:wH}
  For every $n \in \nat$, let $X_n$ be the subposet of Knijnenburg's
  dcpo $X$ consisting of $a_0$, \ldots, $a_n$, $b_0$, \ldots, $b_n$
  and every $c_k$, $k \in \nat$, see the left part of
  Figure~\ref{fig:knijnenburg:limproj}.  Let $p_n \colon X \to X_n$
  map each $a_k$ to $a_{\min (k,n)}$, $b_k$ to $b_{\min (k, n)}$, and
  each $c_k$ to itself.  Each space $X_n$ is a dcpo, and is given the
  Scott topology.  It is easy to see that every directed family in
  $X_n$ must be trivial, in the sense that it already contains its
  supremum.  It follows that the topology of $X_n$ coincides with the
  Alexandroff topology.

  As in every poset with the Alexandroff topology, the irreducible
  closed subsets are exactly the directed subsets, see
  \cite{Hoffmann:Idl:S} or \cite[Fact 8.2.49]{JGL-topology}; since
  $X_n$ is a dcpo, every irreducible closed subset has a largest
  element, so $X_n$ is sober.  Since each $X_n$ has an Alexandroff
  topology, $X_n$ is weakly Hausdorff.  $X_n$ is compact, too, being
  the upward closure of the finite set $\{a_0, b_0\}$.  But no $X_n$
  is coherent, since
  $\upc a_0 \cap \upc b_0 = \{c_k \mid k \in \nat\}$ is not compact.

  Finally, for all $m \leq n$, let $p_{mn} \colon X_n \to X_m$ be the
  restriction of $p_m$ to $X_n$ and let $e_{mn}$ be the inclusion of
  $X_m$ into $X_n$, $e_n$ be the inclusion of $X_n$ into $X$.  Those
  are continuous maps; a simplifying remark, which applies to all
  except $p_m$, is that any monotonic map from a poset with an
  Alexandroff topology to another poset is continuous.  Then the
  spaces $X_n$ ($n \in \nat$) and the pairs $(e_{mn}, p_{mn})$ with
  $m \leq n \in \nat$ form an ep-system, and $X, {(p_n)}_{n \in \nat}$
  is a projective limit of the projective system
  ${(p_{mn} \colon X_n \to X_m)}_{m \leq n \in \nat}$.
\end{example}

\begin{corollary}
  \label{corl:wH:proj}
  The classes of weakly Hausdorff spaces, of weakly Hausdorff compact
  spaces, of weakly Hausdorff sober spaces, and of weakly Hausdorff
  compact sober spaces are not $\omega$-projective.  Even
  $\omega$-projective limits of ep-systems of such spaces may fail to
  be weakly Hausdorff.
\end{corollary}

\begin{example}
  \label{exa:knijnenburg:coh}
  Let again $X$ be Knijnenburg's dcpo, and let $X'_n$ contain the same
  points as the dcpos $X_n$ of Example~\ref{exa:knijnenburg:wH}, but
  with the following modified ordering $\leq'$: $x \leq' y$ if and
  only if $x \leq y$ (where $\leq$ is the ordering on $X_n$) or
  $x=c_i$ and $y=c_j$ for some $i, j \geq n$ such that $i \leq j$.  We
  give $X'_n$ the Scott topology.

  Note that there is a chain $c_n \leq c_{n+1} \leq \cdots$ in $X'_n$
  without a supremum.  Hence $X'_n$ is not a dcpo.  As such, $X'_n$
  cannot be sober, since otherwise it would be a monotone convergence
  space, namely a $T_0$ space whose topology is coarser than the Scott
  topology of its specialization ordering and that would be a dcpo
  with respect to that specialization ordering, see \cite[Proposition
  8.2.34]{JGL-topology}.  The upward-closure $\upc x$ of any point
  $x \in X'_n$ is Scott-open, and therefore those sets form a base of
  the Scott topology, which then coincides with the Alexandroff
  topology.  In particular, each $X'_n$ is weakly Hausdorff.  Every
  upwards-closed subset of $X'_n$ is the upward closure $\upc A$ of a
  finite set $A$, hence is compact.  It immediately follows that
  $X'_n$ is coherent, and compact.

  With $p_m$, $e_m$, $p_{mn}$, $e_{mn}$ defined as in
  Example~\ref{exa:knijnenburg:wH}, we check that the spaces $X'_n$
  ($n \in \nat$) and the pairs $(e_{mn}, p_{mn})$ with
  $m \leq n \in \nat$ form an ep-system, and $X, {(p_n)}_{n \in \nat}$
  is a projective limit of the projective system
  ${(p_{mn} \colon X'_n \to X'_m)}_{m \leq n \in \nat}$.
\end{example}

\begin{corollary}
  \label{corl:wH:proj:coh}
  The classes of coherent weakly Hausdorff spaces and of compact,
  coherent weakly Hausdorff spaces are not $\omega$-projective.  Even
  $\omega$-projective limits of ep-systems of such spaces may fail to
  be weakly Hausdorff.
\end{corollary}

\section{Conclusion}
\label{sec:conclusion}

One can sum up the essential results of this paper in the following
diagram, listing eight classes of \emph{sober} spaces (with a $1$ in each
coordinate meaning that the corresponding property is to be considered):
\begin{center}
  \includegraphics[scale=0.4]{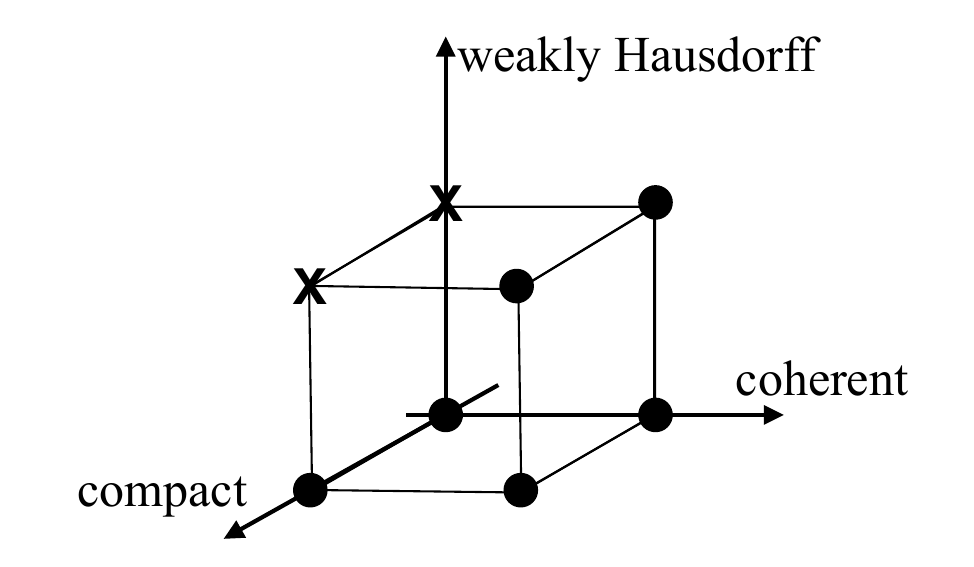}
\end{center}
A blob marks that the corresponding class is projective, and a cross
marks that it is not $\omega$-projective.  The corresponding diagram
for spaces that are not necessarily sober would have crosses
everywhere except at the origin, the diagram for locally compact
spaces would have crosses everywhere.

As a consequence, for example, the best we can say about projective
limits of stably (locally) compact spaces is that they are (locally)
strongly sober, but not more.

There are of course infinitely many classes for which we could ask the
same questions.  Proposition~9.5 of \cite{JGL:kolmogorov} states that
de Brecht's class of quasi-Polish spaces \cite{deBrecht:qPolish} is
$\omega$-projective.  Those generalize both the Polish spaces from
topological measure theory and descriptive set theory, and the
$\omega$-continuous domains from domain theory.  Natural extensions of
the concept include the domain-complete and the LCS-complete spaces of
\cite{dBGLJL:LCScomplete}, and Chen's countably correlated spaces
\cite{Chen:qPolish}.  It is unknown whether any of those three classes
is ($\omega$-)projective.








\section*{Competing interests}

The author declares none.

\bibliographystyle{elsarticle-harv}
\ifarxiv

\else
\bibliography{projlim}
\fi







\end{document}
